\documentclass[11pt,reqno,sumlimits]{amsart}
\usepackage{amsfonts, amsmath, amscd, amssymb, euscript, amsthm, array, booktabs,
dcolumn, shortvrb, tabularx, units, url}
\usepackage[pdftex]{graphicx}
\usepackage[pdftex]{hyperref}
\usepackage[all]{xy}
\normalsize

\DeclareMathOperator{\topo}{top}

\DeclareMathOperator{\todd}{Todd}
\DeclareMathOperator{\BS}{BS}
\DeclareMathOperator{\an}{an}
\DeclareMathOperator{\ind}{ind}
\DeclareMathOperator{\LL}{L}

\DeclareMathOperator{\FL}{FL}

\DeclareMathOperator{\exact}{exact}

\DeclareMathOperator{\dr}{dr}

\DeclareMathOperator{\spin}{spin}

\DeclareMathOperator{\odd}{odd}
\DeclareMathOperator{\even}{even}
\DeclareMathOperator{\im}{Im}

\DeclareMathOperator{\ch}{ch}

\DeclareMathOperator{\U}{U}
\DeclareMathOperator{\BU}{BU}

\DeclareMathOperator{\ho}{Hom}
\DeclareMathOperator{\str}{Str}

\DeclareMathOperator{\CS}{CS}

\begin{document}
\setlength{\baselineskip}{1.0\baselineskip}
\newtheorem{defi}{Definition}
\newtheorem{remark}{Remark}
\newtheorem{coro}{Corollary}
\newtheorem{exam}{Example}
\newtheorem{thm}{Theorem}
\newtheorem{prop}{Proposition}
\newcommand{\wt}[1]{{\widetilde{#1}}}
\newcommand{\ov}[1]{{\overline{#1}}}
\newcommand{\wh}[1]{{\widehat{#1}}}
\newcommand{\poin}{Poincar$\acute{\textrm{e }}$}
\newcommand{\deff}[1]{{\bf\emph{#1}}}
\newcommand{\boo}[1]{\boldsymbol{#1}}
\newcommand{\abs}[1]{\lvert#1\rvert}
\newcommand{\norm}[1]{\lVert#1\rVert}
\newcommand{\inner}[1]{\langle#1\rangle}
\newcommand{\poisson}[1]{\{#1\}}
\newcommand{\biginner}[1]{\Big\langle#1\Big\rangle}
\newcommand{\set}[1]{\{#1\}}
\newcommand{\Bigset}[1]{\Big\{#1\Big\}}
\newcommand{\BBigset}[1]{\bigg\{#1\bigg\}}
\newcommand{\dis}[1]{$\displaystyle#1$}
\newcommand{\R}{\mathbb{R}}
\newcommand{\N}{\mathbb{N}}
\newcommand{\Z}{\mathbb{Z}}
\newcommand{\Q}{\mathbb{Q}}
\newcommand{\E}{\mathcal{E}}
\newcommand{\G}{\mathcal{G}}
\newcommand{\F}{\mathcal{F}}
\newcommand{\V}{\mathcal{V}}
\newcommand{\W}{\mathcal{W}}
\newcommand{\SSS}{\mathcal{S}}
\newcommand{\h}{\mathbb{H}}
\newcommand{\g}{\mathfrak{g}}
\newcommand{\C}{\mathbb{C}}
\newcommand{\A}{\mathcal{A}}
\newcommand{\M}{\mathcal{M}}
\newcommand{\HH}{\mathcal{H}}
\newcommand{\D}{\mathcal{D}}
\newcommand{\PP}{\mathcal{P}}
\newcommand{\K}{\mathcal{K}}
\newcommand{\RR}{\mathcal{R}}
\newcommand{\RRR}{\mathscr{R}}
\newcommand{\DDD}{\mathscr{D}}
\newcommand{\so}{\mathfrak{so}}
\newcommand{\gl}{\mathfrak{gl}}
\newcommand{\aaa}{\mathbb{A}}
\newcommand{\bbb}{\mathbb{B}}
\newcommand{\DD}{\mathsf{D}}
\newcommand{\ccc}{\bold{c}}
\newcommand{\sss}{\mathbb{S}}
\newcommand{\cdd}[1]{\[\begin{CD}#1\end{CD}\]}
\normalsize
\title[Differential Grothendieck--Riemann--Roch Theorem]{A condensed proof
of the differential Grothendieck--Riemann--Roch theorem}
\author{Man-Ho Ho}
\address{Department of Mathematics and Statistics\\ Boston University}
\email{homanho@bu.edu}
\curraddr{Department of Mathematics\\ Hong Kong Baptist University}
\email{homanho@hkbu.edu.hk}
\maketitle
\nocite{*}
\begin{center}
\emph{\small Dedicated to my father Kar-Ming Ho}
\end{center}
\begin{abstract}
We give a direct proof that the Freed--Lott differential analytic index is well defined
and a condensed proof of the differential Grothendieck--Riemann--Roch theorem. As a
byproduct we also obtain a direct proof that the $\R/\Z$ analytic index is well defined
and a condensed proof of the $\R/\Z$ Grothendieck--Riemann--Roch theorem.
\end{abstract}
\tableofcontents
\section{Introduction}

Differential $K$-theory, the differential extension of topological $K$-theory, has
been studied intensively in the last decade. Basically, a differential $K$-theory
class consists of an equivalence class $[E, h, \nabla, \phi]$ of a Hermitian bundle
with connection and a differential form, with the connection and form related
nontrivially.

The mathematical motivation for differential $K$-theory can be traced to Cheeger--Simons
differential characters \cite{CS85}, the unique differential extension of ordinary
cohomology \cite{SS08a}, and to work of Karoubi \cite{K86}. It is thus natural to
look for differential extensions of generalized cohomology theories such as topological
$K$-theory. Various definitions of differential $K$-theory have been given, notably by
Bunke--Schick \cite{BS09}, Freed--Lott \cite{FL10}, Hopkins--Singer \cite{HS05} and
Simons--Sullivan \cite{SS10}. By work of \cite{BS10a}, these models of differential
$K$-theory are all isomorphic. For a detailed survey of differential $K$-theory, see
\cite{BS10}.

The Atiyah-Singer family index theorem can be formulated as the equality of the
analytic and topological pushforward maps
$$\ind^{\an}=\ind^{\topo}:K(X)\to K(B).$$
Applying the Chern character, we get the Grothendieck--Riemann--Roch theorem, the
commutativity of
\cdd{K(X) @>\ch>> H^{\even}(X; \Q) \\ @V\ind^{\an}VV @VV\int_{X/B}\todd(X/B)\cup(\cdot) V
\\ K(B) @>>\ch> H^{\even}(B; \Q)}

Analogous theorems hold in differential $K$-theory. Bunke--Schick prove the differential
Grothendieck--Riemann--Roch theorem (dGRR) \cite[Theorem 6.19]{BS09}, i.e., for a proper
submersion $\pi:X\to B$ of even relative dimension, the following diagram is commutative:
\begin{equation}\label{eqspdgrr 1}
\begin{CD}
\wh{K}_{\BS}(X) @>\wh{\ch}_{\BS}>> \wh{H}^{\even}(X; \R/\Q) \\ @V\ind^{\an}_{\BS}VV
@VV\wh{\int_{X/B}}\wh{\todd}(\wh{\nabla}^{T^VX})\ast(\cdot) V \\ \wh{K}_{\BS}(B)
@>>\wh{\ch}_{\BS}> \wh{H}^{\even}(B; \R/\Q)
\end{CD}
\end{equation}
where $\wh{H}(X; \R/\Q)$ is the ring of differential characters, $\wh{\ch}_{\BS}$ is
the Bunke--Schick differential Chern character \cite[\S 6.2]{BS09}, $\ind^{\an}_{\BS}$
is the Bunke--Schick differential analytic index \cite[\S 3]{BS09} and
\dis{\wh{\int_{X/B}}\wh{\todd}(\wh{\nabla}^{T^VX})\ast} is a modified pushforward of
differential characters \cite[\S 6.4]{BS09}. The notation is explained more fully in
later sections. On the other hand, Freed--Lott prove the differential family index
theorem
\cite[Theorem 7.32]{FL10}
$$\ind^{\an}_{\FL}=\ind^{\topo}_{\FL}:\wh{K}_{\FL}(X)\to\wh{K}_{\FL}(B),$$
where $\ind^{\an}_{\FL}$ and $\ind^{\topo}_{\FL}$ are the Freed--Lott differential
analytic index \cite[Definition 3.11]{FL10} and the differential topological index
\cite[Definition 5.33]{FL10}. Applying the differential Chern character
$\wh{\ch}_{\FL}$ yields the dGRR \cite[Corollary 8.23]{FL10}. Since $\ind^{\an}_{\BS}=
\ind^{\an}_{\FL}$ \cite[Corollary 5.5]{BS09}, the two dGRR theorems are essentially
the same.

Both proofs of the dGRR are involved, and yield much more information than the dGRR
alone. In particular, the fact that $\ind^{\an}_{\FL}$ is well defined follows {\it a
posteriori} from the differential family index theorem. The main results of this paper
are a direct proof that $\ind^{\an}_{\FL}$ is well defined and a condensed proof of
dGRR. Note that without the direct proof that $\ind^{\an}_{\FL}$ is well defined we
cannot compute $\wh{\ch}_{\FL}(\ind^{\an}_{\FL}(\E))$ without using the differential
family index theorem.

We first prove these theorems in the special case where the family of kernels of the
Dirac operators has constant dimension, i.e., $\ker(\DD^E)\to B$ is a superbundle.
The proof of $\ind^{\an}_{\FL}$ is well defined makes use of the variational formula
of the Bismut-Cheeger eta form, and the proof of the dGRR relies on a result of Bismut
\cite[Theorem 1.15]{B05}, which allows us to shorten the existing proofs at the
expense of using this theorem. The general case follows from a standard perturbation
argument as in \cite[\S 7]{FL10}. It is stated in \cite[p. 23]{B05} that Bismut's
theorem extends to the general case.

J. Lott proves the equality \cite[Corollary 3]{L94}
$$\ind^{\an}_{\LL}=\ind^{\topo}_{\LL}:K^{-1}_{\LL}(X; \R/\Z)\to K^{-1}_{\LL}(B; \R/\Z)$$
of an analytic and topological index in his geometric model of $K^{-1}(X; \R/\Z)$. 
This index theorem and the corresponding GRR theorem are consequences of the Freed-Lott 
differential family index theorem. Thus we also obtain a direct proof that the analytic 
index $\ind^{\an}_{\LL}$ is well defined, and a condensed proof of the corresponding 
GRR theorem. Indeed, Bismut already states in \cite[p.17]{B05} (without proof) that
\cite[Theorem 1.15]{B05} implies this GRR theorem. For proofs of these theorems without 
using differential $K$-theory, see \cite{H12u}.

The next two sections contain the necessary background material. Section 2 reviews
Cheeger--Simons differential characters, their multiplication and some properties of
pushforward. Section 3 reviews Freed--Lott differential $K$-theory, the construction of
the Freed--Lott differential analytic index and the Freed--Lott differential Chern
character. The main results of the paper are proved in Section 4.
\section*{Acknowledgement}
The author would like to thank several people. First of all we would like to thank
Steven Rosenberg for suggesting this problem and many stimulating discussions. Second
we would like to thank Bai-Ling Wang for his comments on the Bismut-Cheeger eta form,
and Ulrich Bunke for kindly pointing out an error in Proposition \ref{prop 1} in a
previous version of this paper. Third, we would like to thank Jean-Michel Bismut
and Sebastian Goette for providing the author many valuable insights about the
variational formula of the Bismut-Cheeger eta form used in the proof of Proposition
\ref{prop 1}. Last but not least we would like to thank the referee for the his
helpful comments.
\section{Cheeger--Simons differential characters}
\subsection{Definition of differential characters}
We recall Cheeger--Simons differential characters \cite{CS85} with coefficients in
$\R/\Q$. Let $X$ be a manifold. The ring of differential characters of degree $k\geq 1$
is
$$\wh{H}^k(X; \R/\Q)=\set{f\in\ho(Z_{k-1}(X), \R/\Q)|\exists\omega_f\in\Omega^k(X)
\textrm{ such that }f\circ\partial=\ov{\omega_f}},$$
where $\ov{\ ^{}}:\Omega^k(X)\to C^k(X; \R/\Q)$ is an injective homomorphism defined
by \dis{\ov{\omega}(c_k):=\int_{c_k}\omega\mod\Q}. It is easy to show that $\omega_f$
is a closed $k$-form with periods in $\Q$ and is uniquely determined by $f\in\wh{H}^k
(X; \R/\Q)$. In the following hexagon, the diagonal sequences are exact, and every
triangle and square commutes \cite[Theorem 1.1]{CS85}:
\begin{equation}\label{eqspdgrr 5}
\xymatrix{\scriptstyle 0 \ar[dr] & \scriptstyle & \scriptstyle & \scriptstyle & \scriptstyle 0
\\ & \scriptstyle H^{k-1}(X; \R/\Q) \ar[rr]^{-B} \ar[dr]^{i_1} & \scriptstyle & \scriptstyle
H^k(X; \Q) \ar[ur] \ar[dr]^r & \scriptstyle \\ \scriptstyle H^{k-1}(X; \R) \ar[ur]^{\alpha}
\ar[dr]_{\beta} & \scriptstyle & \scriptstyle\wh{H}^k(X; \R/\Q) \ar[ur]^{\delta_2}
\ar[dr]^{\delta_1} & \scriptstyle & \scriptstyle H^k(X; \R) \\ \scriptstyle & \scriptstyle \frac{\Omega^{k-1}(X)}{\Omega^{k-1}_{\Q}(X)} \ar[rr]_{d} \ar[ur]^{i_2} & \scriptstyle &
\scriptstyle\Omega^k_{\Q}(X) \ar[ur]^s \ar[dr] & \scriptstyle \\ \scriptstyle 0 \ar[ur] &
\scriptstyle & \scriptstyle & \scriptstyle & \scriptstyle 0}
\end{equation}
The maps are defined as follows: $r$ is induced by $\Q\hookrightarrow\R$,
$$i_1([z])=z|_{Z_{k-1}(X)},~~~i_2(\omega)=\ov{\omega}|_{Z_{k-1}(X)},~~~\delta_1(f)
=\omega_f\textrm{ and }\delta_2(f)=[c],$$
where $[c]\in H^k(X; \Q)$ is the unique cohomology class satisfying $r[c]=[\omega_f]$,
and $\Omega^k_\Q(X)$ consists of closed forms with periods in $\Q$. (We will not use
the other maps.) The character diagram uniquely characterizes differential extension
of ordinary cohomology \cite{SS08a}.

Invariant polynomials for $\U(n)$ have associated characteristic classes and differential
characters. In particular, for a Hermitian vector bundle $E\to X$ with a metric $h$ and
a unitary connection $\nabla$, the differential Chern character is the unique natural
differential character \cite[Theorem 2.2]{CS85}
\begin{equation}\label{eqspdgrr 6}
\wh{\ch}(E, h, \nabla)\in\wh{H}^{\even}(X; \R/\Q)
\end{equation}
such that
$$\delta_1(\wh{\ch}(E, h, \nabla))=\ch(\nabla)\textrm{ and } \delta_2(\wh{\ch}
(E, h, \nabla))=\ch(E).$$
We will write $\wh{\ch}(E, h, \nabla)$ as $\wh{\ch}(E, \nabla)$ in the sequel.
\subsection{Multiplication of differential characters}
In \cite{CS85} multiplication of differential characters is defined. Let
$E:\Omega^{k_1}(X)\times\Omega^{k_2}(X)\to C^{k_1+k_2-1}(X; \R)$ be a natural chain
homotopy between the wedge product $\wedge$ and the cup product $\cup$, i.e., for
$\omega_i\in\Omega^{k_i}(X)$, we have
$$\delta E(\omega_1, \omega_2)+E(d\omega_1, \omega_2)+(-1)^{k_1}E(\omega_1, d\omega_2)
=\omega_1\wedge\omega_2-\omega_1\cup\omega_2$$
as cochains. Note that any two choices of $E$ are naturally chain homotopic. For
$f\in\wh{H}^{k_1}(X; \R/\Q)$ and $g\in\wh{H}^{k_2}(X; \R/\Q)$, define $f\ast g\in
\wh{H}^{k_1+k_2}(X; \R/\Q)$ by
$$f\ast g=\big(\wt{T_f\cup\omega_g}+(-1)^{k_1}\wt{\omega_f\cup T_g}+\wt{T_f\cup\delta
T_g}+\wt{E(\omega_f, \omega_g)}\big)|_{Z_{k_1+k_2-1}(X)},$$
where $T_f, T_g\in C^{k-1}(X, \R)$ are lifts of $f$ and $g$.
\begin{prop}\cite[Theorem 1.11]{CS85}\label{propspdgrr 1}
Let $f\in\wh{H}^{k_1}(X; \R/\Q)$ and $g\in\wh{H}^{k_2}(X; \R/\Q)$.
\begin{enumerate}
  \item $f\ast g$ is independent of the choice of the lifts $T_f$ and $T_g$,
  \item $f\ast(g\ast h)=(f\ast g)\ast h$ and $f\ast g=(-1)^{k_1k_2}g\ast f$,
  \item $\omega_{f\ast g}=\omega_f\wedge\omega_g$ and $c_{f\ast g}=c_f\cup c_g$. i.e.,
        $\delta_1$ and $\delta_2$ are ring homomorphisms,
  \item If $\phi:N\to M$ is a smooth map, then $\phi^*(f\ast g)=\phi^*(f)\ast\phi^*(g)$,
  \item If $\theta\in\Omega^\bullet(X)$, then $i_2(\theta)\ast f=i_2(\theta\wedge\omega_f)$,
  \item If $[c]\in H^\bullet(X; \R/\Q)$, then $f\ast i_1([c])=(-1)^{k_1}i_1([c_f]\cup[c])$.
\end{enumerate}
\end{prop}
\subsection{Pushforward of differential characters}
The pushforward of differential characters is defined in \cite[\S 3.4]{HS05}. We only
consider proper submersions $\pi:X\to B$ with closed fibers of relative dimension $n$,
where the definition \cite[\S 8.3]{FL10} is straightforward: for $k\geq n$,
$$\wh{\int_{X/B}}:\wh{H}^k(X; \R/\Q)\to\wh{H}^{k-n}(B;
\R/\Q), \ \left(\wh{\int_{X/B}}f\right) (z) = f(\pi^{-1}(z)).$$
Let \dis{\int_{X/B}} denote both the pushforward of forms and cohomology classes.
\begin{prop}\cite[\S 3.4]{HS05}\label{propspdgrr 2}
Let $f\in\wh{H}^k(X; \R/\Q)$, $[c]\in H^{k-1}(X; \R/\Q)$ and
\dis{\theta\in\frac{\Omega^{k-1}(X)}{\Omega^{k-1}_\Q(X)}}. Then
\begin{enumerate}
  \item \dis{\delta_1\bigg(\wh{\int_{X/B}}f\bigg)=\int_{X/B}\omega_f}.
  \item \dis{\delta_2\bigg(\wh{\int_{X/B}}f\bigg)=\int_{X/B}[c_f]}.
  \item \dis{\wh{\int_{X/B}}i_1([c])=i_1\bigg(\int_{X/B}[c]\bigg)}.
  \item \dis{\wh{\int_{X/B}}i_2(\theta)=i_2\bigg(\int_{X/B}\theta\bigg)}.
\end{enumerate}
\end{prop}
\section{Freed--Lott differential $K$-theory}
\subsection{Definition of Freed--Lott differential $K$-theory}

In this subsection we review Freed--Lott differential $K$-theory \cite{FL10}.

The Freed--Lott differential $K$-group $\wh{K}_{\FL}(X)$ is the abelian group generated
by quadruples $\E=(E, h, \nabla, \phi)$, where $(E,h,\nabla)\to X$ is a complex vector
bundle with a hermitian metric $h$ and a unitary connection $\nabla$, and
\dis{\phi\in\frac{\Omega^{\odd}(X)}{\im(d)}}. The only relation is $\E_1=\E_2$ if and
only if there exists a generator $(F, h^F, \nabla^F, \phi^F)$ of $\wh{K}_{\FL}(X)$ such
that $E_1\oplus F\cong E_2\oplus F$ and $\phi_1-\phi_2=\CS(\nabla^{E_2}\oplus\nabla^F,
\nabla^{E_1}\oplus\nabla^F)$.

In the following hexagon, the diagonal sequences are exact, and every triangle and
square commutes \cite{FL10}:
\begin{equation}\label{eqspdgrr 20}
\xymatrix{\scriptstyle 0 \ar[dr] & \scriptstyle & \scriptstyle & \scriptstyle &
\scriptstyle 0 \\ & \scriptstyle K^{-1}_{\LL}(X; \R/\Z) \ar[rr]^{-B} \ar[dr]^{i} & 
\scriptstyle & \scriptstyle K(X; \Z) \ar[ur] \ar[dr]^{\ch_\R} & \scriptstyle \\ 
\scriptstyle H^{\odd}(X; \R) \ar[ur]^{\alpha} \ar[dr]_{\beta} & \scriptstyle & 
\scriptstyle\wh{K}_{\FL}(X) \ar[ur]^{\delta} \ar[dr]^{\ch_{\wh{K}_{\FL}}} & \scriptstyle 
& \scriptstyle H^{\even}(X; \R) \\ \scriptstyle & \scriptstyle \frac{\Omega^{\odd}
(X)}{\Omega^{\odd}_{\BU}(X)} \ar[rr]_{d} \ar[ur]^{j} & \scriptstyle & \scriptstyle
\Omega^{\even}_{\BU}(X) \ar[ur]^{\dr} \ar[dr] & \scriptstyle \\ \scriptstyle 0 \ar[ur] 
& \scriptstyle & \scriptstyle & \scriptstyle & \scriptstyle 0}
\end{equation}
where $\ch_\R:=r\circ\ch:K(X)\to H^{\even}(X; \R)$,
$$\Omega^{\bullet}_{\BU}(X):=\set{\omega\in\Omega^{\bullet}_{d=0}(X)|[\omega]\in\im
(\ch^\bullet:K^{-(\bullet\mod 2)}\to H^\bullet(X; \Q))},$$
where $\bullet\in\set{\even, \odd}$. The maps are defined as follows:
$$\delta(\E)=[E],~~~~\ch_{\wh{K}_{\FL}}(\E)=\ch(\nabla)+d\phi,~~~~j(\phi)=(0, 0, d,
\phi),$$
$i$ is the natural inclusion map, $\dr$ is the de Rham map, and the sequence of maps
$(\alpha, \beta, \ch_\R)$ can be regarded as the Bockstein sequence in $K$-theory as
we may identify $H^\bullet(X; \R)$ with $K^\bullet(X; \R)$ via the Chern character.

The Freed--Lott differential Chern character $\wh{\ch}_{\FL}:\wh{K}_{\FL}(X)\to
\wh{H}^{\even}(X; \R/\Q)$ is defined by
$$\wh{\ch}_{\FL}(\E)=\wh{\ch}(E, \nabla)+i_2(\phi),$$
where $\wh{\ch}(E, \nabla)$ is given in (\ref{eqspdgrr 6}), and $i_2$ is in
(\ref{eqspdgrr 5}).
\subsection{Freed--Lott differential analytic index}

The Freed--Lott differential analytic index of a generator $\E=(E, h, \nabla^E, \phi)
\in\wh{K}_{\FL}(X)$ is roughly given by the geometric construction of the
analytic index of $(E, h, \nabla)$ with a modified pushforward of the form $\phi$.

In more detail, let $\pi:X\to B$ be a proper submersion of even relative dimension $n$,
and let $T^VX\to X$ be the vertical tangent bundle, which is assumed to have a metric
$g^{T^VX}$. A given horizontal distribution $T^HX\to X$ and a Riemannian metric $g^{TB}$
on $B$ determine a metric on $TX\to X$  by $g^{TX}:=g^{T^VX}\oplus\pi^*g^{TB}$. If
$\nabla^{TX}$ is the corresponding Levi-Civita connection, then $\nabla^{T^VX}:=P\circ
\nabla^{TX}\circ P$ is a connection on $T^VX\to X$, where $P:TX\to T^VX$ is the
orthogonal projection. $T^VX\to X$ is assumed to have a $\spin^c$ structure. Denote by
$S^VX\to X$ the spinor$^c$ bundle associated to the characteristic Hermitian line bundle
$L^VX\to X$ with a unitary connection $\nabla^{L^VX}$. Define a connection
$\wh{\nabla}^{T^VX}$ on $S^VX\to X$ by $\wh{\nabla}^{T^VX}:=\nabla^{T^VX}\otimes
\nabla^{L^V}$, where $\nabla^{T^VX}$ also denotes the lift of $\nabla^{T^VX}$ to the
local spinor bundle. The Todd form $\todd(\wh{\nabla}^{T^VX})$ of $S^VX\to X$ is
defined by
$$\todd(\wh{\nabla}^{T^VX}):=\wh{A}(\nabla^{T^VX})\wedge e^{\frac{1}{2}c_1
(\nabla^{L^VX})}.$$
For $k\geq n$, the modified pushforward of forms $\pi_*:\Omega^k(X)\to\Omega^{k-n}(B)$
\cite[(3.2)]{FL10} is defined by
$$\pi_*(\phi)=\int_{X/B}\todd(\wh{\nabla}^{T^VX})\wedge\phi.$$
It induces a map, still denoted by \dis{\pi_*:\frac{\Omega^{\odd}(X)}{\im(d)}\to
\frac{\Omega^{\odd}(B)}{\im(d)}}.

We briefly recall the definition of the Bismut--Cheeger eta form \dis{\wt{\eta}(\E)\in
\frac{\Omega^{\odd}(B)}{\im(d)}} associated to $\E\in\wh{K}_{\FL}(X)$. With the above
setup, consider the infinite-rank superbundle $\pi_*E\to B$, where the fibers at each
$b\in B$ is given by
$$(\pi_*E)_b:=\Gamma(X_b, (S^VX\otimes E)|_{X_b}).$$
Recall that $\pi_*E\to B$ admits an induced Hermitian metric and a connection
$\nabla^{\pi_*E}$ compatible with the metric \cite[\S 9.2, Proposition 9.13]{BGV}. For
each $b\in B$, the canonically constructed Dirac operator
$$\DD^E_b:\Gamma(X_b, (S^VX\otimes E)|_{X_b})\to\Gamma(X_b, (S^VX\otimes E)|_{X_b})$$
gives a family of Dirac operators, denoted by $\DD^E:\Gamma(X, S^VX\otimes E)\to\Gamma
(X, S^VX\otimes E)$. Assume the family of kernels $\ker(\DD^E_b)$ has locally constant
dimension, i.e., $\ker(\DD^E)\to B$ is a finite-rank Hermitian superbundle. Let
$P:\pi_*E\to\ker(\DD^E)$ be the orthogonal projection, $h^{\ker(\DD^E)}$ be the
Hermitian metric on $\ker(\DD^E)\to B$ induced by $P$, and
$\nabla^{\ker(\DD^E)}:=P\circ\nabla^{\pi_*E}\circ P$ be the connection on $\ker(\DD^E)
\to B$ compatible to $h^{\ker(\DD^E)}$.

The (scaled) Bismut-superconnection $\aaa_t:\Omega(B, \pi_*E)\to\Omega(B, \pi_*E)$
\cite[Definition 3.2]{B86} (see also \cite[Proposition 10.15]{BGV} and \cite[(1.4)]{D91}),
is defined by
$$\aaa_t:=\sqrt{t}\DD^E+\nabla^{\pi_*E}-\frac{c(T)}{4\sqrt{t}},$$
where $c(T)$ is the Clifford multiplication by the curvature 2-form of the fiber bundle.
The Bismut--Cheeger eta form $\wt{\eta}(\E)$ \cite[(2.26)]{BC89} (see also \cite{D91}
and \cite[Theorem 10.32]{BGV}) is defined by
$$\wt{\eta}(\E):=\frac{1}{2\sqrt{\pi}}\int^\infty_0\frac{1}{\sqrt{t}}\str\bigg
(\frac{d\aaa_t}{dt}e^{-\aaa_t^2}\bigg)dt.$$
It satisfies
$$d\wt{\eta}(\E)=\int_{X/B}\todd(\wh{\nabla}^{T^VX})\wedge\ch(\nabla^E)-\ch(\nabla^{\ker
(\DD^E)}).$$
The Freed--Lott differential analytic index $\ind^{\an}_{\FL}:\wh{K}_{\FL}(X)\to
\wh{K}_{\FL}(B)$ is
$$\ind^{\an}_{\FL}(\E)=(\ker(\DD^E), h^{\ker(\DD^E)}, \nabla^{\ker(\DD^E)}, \pi_*(\phi)+
\wt{\eta}(\E)).$$
\section{Main Results}
In this section we give a direct proof that $\ind^{\an}_{\FL}$ is well defined and
give a condensed proof of the dGRR. We first recall a theorem of Bismut \cite{B05}.
In the setup of \S3.2, with the fibers $\spin$ and $\ker(\DD^E)\to B$ assumed to form
a superbundle, we have
\begin{equation}\label{eqspdgrr 8}
\wh{\ch}(\ker(\DD^E), \nabla^{\ker(\DD^E)})+i_2(\wt{\eta}(\E))=\wh{\int_{X/B}}
\wh{\wh{A}}(T^VX, \nabla^{T^VX})\ast\wh{\ch}(E, \nabla^E)
\end{equation}
\cite[Theorem 1.15]{B05}. If the fibers are only $\spin^c$, (\ref{eqspdgrr 8})
has the obvious modification
\begin{equation}\label{eqspdgrr 9}
\wh{\ch}(\ker(\DD^E), \nabla^{\ker(\DD^E)})+i_2(\wt{\eta}(\E))=\wh{\int_{X/B}}\wh{\todd}
(T^VX, \wh{\nabla}^{T^VX})\ast\wh{\ch}(E, \nabla^E),
\end{equation}
for $\wh{\todd}(T^VX, \wh{\nabla}^{T^VX})\in\wh{H}^{\even}(X; \R/\Q)$ the differential
character associated to the Todd form and the Todd class as in (\ref{eqspdgrr 6}), and
similarly for $\wh{\wh{A}}(T^VX, \nabla^{T^VX})$. We will write $\wh{\todd}(T^VX,
\wh{\nabla}^{T^VX})$ as $\wh{\todd}(\wh{\nabla}^{T^VX})$ in the sequel. Note that
(\ref{eqspdgrr 8}) and (\ref{eqspdgrr 9}) extend to the general case where $\ker(\DD^E)
\to B$ does not form a bundle \cite[p.23]{B05}.
\subsection{Freed--Lott differential analytic index}
In this subsection we prove that $\ind^{\an}_{\FL}$ is well defined.
\begin{prop}\label{prop 1}
Let $\pi:X\to B$ be a proper submersion with closed $\spin^c$ fibers of even relative
dimension. If $\E=\F\in\wh{K}_{\FL}(X)$, then
$$\ind^{\an}_{\FL}(\E)=\ind^{\an}_{\FL}(\F).$$
\end{prop}
\begin{proof}
Let $f:=\ind^{\an}_{\FL}(\E)-\ind^{\an}_{\FL}(\F)$. Since there exists a generator $\G$
in $\wh{K}_{\FL}(X)$ such that $E\oplus G\cong F\oplus G$ and $\phi^E-\phi^F=\CS(\nabla^F
\oplus\nabla^G, \nabla^E\oplus\nabla^G)$ up to an exact form, it follows that $\ker(\DD^E)
\oplus\ker(\DD^G)\cong\ker(\DD^F)\oplus\ker(\DD^G)$ and therefore $\delta(f)=[\ker(\DD^E)]
-[\ker(\DD^F)]=0$ in $K(B)$. By (\ref{eqspdgrr 20}) there exists a unique \dis{\omega\in
\frac{\Omega^{\odd}(B)}{\Omega^{\odd}_{\BU}(B)}} such that $j(\omega)=f$. Since
\begin{displaymath}
\begin{split}
&~~~~\bigg(\ker(\DD^E), h^{\ker(\DD^E)}, \nabla^{\ker(\DD^E)}, \wt{\eta}(\E)+\int_{X/B}
\todd(\wh{\nabla}^{T^VX})\wedge\phi^E\bigg)\\
&=\bigg(\ker(\DD^F), h^{\ker(\DD^F)}, \nabla^{\ker(\DD^F)}, \wt{\eta}(\F)+\int_{X/B}
\todd(\wh{\nabla}^{T^VX})\wedge\phi^F\bigg)+j(\omega),
\end{split}
\end{displaymath}
it follows that, up to an exact form
\begin{displaymath}
\begin{split}
&~~~~\CS(\nabla^{\ker(\DD^E)}\oplus\nabla^{\ker(\DD^G)}, \nabla^{\ker(\DD^F)}\oplus
\nabla^{\ker(\DD^G)})\\
&=\CS(\nabla^{\ker(\DD^E)}\oplus\nabla^{\ker(\DD^G)}, \nabla^{\ker(\DD^F)}\oplus
\nabla^{\ker(\DD^G)}\oplus d)\\
&=\wt{\eta}(\F)+\int_{X/B}\todd(\wh{\nabla}^{T^VX})\wedge\phi^F+\omega-\wt{\eta}(\E)
-\int_{X/B}\todd(\wh{\nabla}^{T^VX})\wedge\phi^E\\
&=\omega+\wt{\eta}(\F)-\wt{\eta}(\E)+\int_{X/B}\todd(\wh{\nabla}^{T^VX})\wedge\CS
(\nabla^E\oplus\nabla^G, \nabla^F\oplus\nabla^G),
\end{split}
\end{displaymath}
and hence
\begin{displaymath}
\begin{split}
\omega&=\CS(\nabla^{\ker(\DD^E)}\oplus\nabla^{\ker(\DD^G)}, \nabla^{\ker(\DD^F)}\oplus
\nabla^{\ker(\DD^G)})+\wt{\eta}(\E)-\wt{\eta}(\F)\\
&~~~~+\int_{X/B}\todd(\wh{\nabla}^{T^VX})\wedge\CS(\nabla^F\oplus\nabla^G, \nabla^E
\oplus\nabla^G)
\end{split}
\end{displaymath}
in \dis{\frac{\Omega^{\odd}(B)}{\Omega^{\odd}_{\BU}(B)}}. We prove that $\omega\in
\Omega^{\odd}_{\BU}(B)$. Since the variational formula of the Bismut-Cheeger eta form
is given by\footnote{I would like to thank J.-M. Bismut and S. Goette for pointing out
this formula.}
\begin{displaymath}
\begin{split}
\wt{\eta}(\F)-\wt{\eta}(\E)&=\CS(\nabla^{\ker(\DD^E)}\oplus\nabla^{\ker(\DD^G)},
\nabla^{\ker(\DD^F)}\oplus\nabla^{\ker(\DD^G)})\\
&~~~~~+\int_{X/B}\todd(\wh{\nabla}^{T^VX})\wedge\CS(\nabla^F\oplus\nabla^G, \nabla^E\oplus
\nabla^G)\mod\Omega^{\odd}_{\exact}(B),
\end{split}
\end{displaymath}
where $\Omega^{\odd}_{\exact}(B)$ denotes the ring of exact odd forms on $B$, it
follows that $\omega\in\Omega^{\odd}_{\exact}(B)\subseteq\Omega^{\odd}_{\BU}(B)$.
Thus $j(\omega)=0$.
\end{proof}
\subsection{Differential Grothendieck--Riemann--Roch theorem}
In this subsection we give a condensed proof of the dGRR.
\begin{thm}\label{thm 1}
Let $\pi:X\to B$ be a proper submersion with closed $\spin^c$ fibers of even relative
dimension. Then the following diagram is commutative:
\cdd{\wh{K}_{\FL}(X) @>\wh{\ch}_{\FL}>> \wh{H}^{\even}(X; \R/\Q) \\ @V\ind^{\an}_{\FL}VV
@VV\wh{\int_{X/B}}\wh{\todd}(\wh{\nabla}^{T^VX})\ast (\cdot) V \\ \wh{K}_{\FL}(B)
@>>\wh{\ch}_{\FL}> \wh{H}^{\even}(B; \R/\Q)}
i.e., for $\E=(E, h, \nabla, \phi)\in\wh{K}_{\FL}(X)$, we have
$$\wh{\ch}_{\FL}(\ind^{\an}_{\FL}(\E))=\wh{\int_{X/B}}\wh{\todd}(\wh{\nabla}^{T^VX})\ast
\wh{\ch}_{\FL}(\E).$$
\end{thm}
\begin{proof}
Observe that
\begin{equation}\label{eqspdgrr 13}
f:=\wh{\ch}_{\FL}(\ind^{\an}_{\FL}(\E))-\wh{\int_{X/B}}\wh{\todd}(\wh{\nabla}^{T^VX})
\ast\wh{\ch}_{\FL}(\E)
\end{equation}
does not depend on the choice of $\phi$ in $\E=(E, h, \nabla, \phi)\in\wh{K}_{\FL}(X)$.
To see this, note that
\begin{equation}\label{eqspdgrr 14}
\wh{\ch}_{\FL}(\ind^{\an}_{\FL}(\E))=\wh{\ch}(\ind^{\an}(E, h, \nabla^E))+
i_2\bigg(\int_{X/B}\todd(\wh{\nabla}^{T^VX})\wedge\phi\bigg)+i_2(\wt{\eta}),
\end{equation}
and
\begin{equation}\label{eqspdgrr 15}
\begin{split}
&\qquad\wh{\int_{X/B}}\wh{\todd}(\wh{\nabla}^{T^VX})\ast\wh{\ch}_{\FL}(\E)\\
&=\wh{\int_{X/B}}\wh{\todd}(\wh{\nabla}^{T^VX})\ast\wh{\ch}(E, \nabla^E)+
\wh{\int_{X/B}}\wh{\todd}(\wh{\nabla}^{T^VX})\ast i_2(\phi).
\end{split}
\end{equation}
Using Prop.~\ref{propspdgrr 1}(2) and Prop.~\ref{propspdgrr 2}(4), (5), we get
\begin{eqnarray}\label{eqspdgrr 16}
i_2\bigg(\int_{X/B}\todd(\wh{\nabla}^{T^VX})\wedge\phi\bigg)
&=& \wh{\int_{X/B}}i_2(\todd(\wh{\nabla}^{T^VX})\wedge\phi)\nonumber\\
&=&\wh{\int_{X/B}}i_2(\phi\wedge\todd(\wh{\nabla}^{T^VX}))\nonumber\\
&=&\wh{\int_{X/B}}i_2(\phi)\ast\wh{\todd}(\wh{\nabla}^{T^VX})\\
&=&\wh{\int_{X/B}}\wh{\todd}(\wh{\nabla}^{T^VX})\ast i_2(\phi).\nonumber
\end{eqnarray}
It follows from (\ref{eqspdgrr 14}), (\ref{eqspdgrr 15}) and (\ref{eqspdgrr 16}) that
\begin{equation}\label{eqspdgrr 17}
f=\wh{\ch}(\ind^{\an}(E, h, \nabla^E))+i_2(\wt{\eta}(\E))-\wh{\int_{X/B}}
\wh{\todd}(\wh{\nabla}^{T^VX})\ast\wh{\ch}(E, \nabla^E).
\end{equation}
Thus proving (\ref{eqspdgrr 13}) is zero is equivalent to proving (\ref{eqspdgrr 17})
is zero, which follows from (\ref{eqspdgrr 9}).
\end{proof}
\bibliographystyle{amsplain}
\bibliography{MBib}
\end{document}